\documentclass{article}
\usepackage[utf8]{inputenc}
\usepackage{amsmath}
\usepackage{amsthm}
\usepackage{amsfonts}
\usepackage{amssymb}
\usepackage[all]{xy}
\usepackage{hyperref}
\usepackage{indentfirst}

\newtheorem{thm}{Theorem}[section]
\newtheorem{prop}[thm]{Proposition}
\newtheorem{lem}[thm]{Lemma}
\newtheorem{theorem}[thm]{Theorem}

\theoremstyle{definition}
\newtheorem{defi}[thm]{Definiton}
\newtheorem{example}[thm]{Example}
\newtheorem{definition}[thm]{Definition}
\newtheorem*{ack}{Acknowledgements}

\title{Supramenable groups and partial actions}
\date{}
\author{Eduardo P. Scarparo\thanks{This work was supported by CNPq, National Council for Scientific and Technological Development - Brazil.}}
\renewcommand\footnotemark{}
\begin{document}
\maketitle

\begin{abstract}
We characterize supramenable groups in terms of existence of invariant probability measures for partial actions on compact Hausdorff spaces and existence of tracial states on partial crossed products.

These characterizations show that, in general, one cannot decompose a partial crossed product of a $\mathrm{C}^*$-algebra by a semi-direct product of groups as two iterated partial crossed products. We give conditions which ensure that such decomposition is possible.
\end{abstract}
\section{Introduction} 
Partial actions of groups on $\mathrm{C}^*$-algebras were introduced by Exel (\cite{exelcircle})  and McClanahan (\cite{mcclanahan}) as a means of computing invariants and describing the structure of $\mathrm{C}^*$-algebras. It has since been also studied in other categories, notably of sets, topological spaces and algebras (see \cite{exel} for details). 

A group $G$ is called \emph{amenable} if it carries an invariant, finitely additive measure $\mu$ such that $\mu(G)=1$. Following Rosenblatt (\cite{rosenblatt1974invariant}), a group $G$ is called \emph{supramenable} if for every non-empty $A\subset G$ there is an invariant, finitely additive measure $\mu$ on $G$ such that $\mu(A)=1$.

The class of supramenable groups is closed under taking subgroups, quotients and direct limits. Abelian groups and, more generally, groups of subexponential growth are supramenable (see \cite[Chapter 12]{s1993banach} for a proof of these facts). It is not known if the direct product of supramenable groups is supramenable and if every supramenable group has subexponential growth.

It is a well-known fact that a group is amenable if and only if whenever it acts on a compact Hausdorff space, then the space admits an invariant probability measure. There is also a non-commutative version of this result which says that a group is amenable if and only if whenever it acts on a unital $\mathrm{C}^*$-algebra which has a tracial state, then the associated crossed product also has a tracial state.

These results do not hold for partial actions, though, as Hopenwasser showed in \cite{hop} that the Cuntz algebras can be realized as partial crossed products associated to partial actions of certain amenable groups on the Cantor set.

In \cite{kmr}, Kellerhals, Monod and Rørdam showed that a group is supramenable if and only if whenever it acts co-compactly on a locally compact Hausdorff space, then the space admits an invariant, non-zero and regular measure. 

The main purpose of this paper is to show that the role of amenable groups is played by supramenable groups in the  context of partial actions. 

In Section \ref{comm}, we prove that a group is supramenable if and only if whenever it partially acts on a compact Hausdorff space, then the space admits an invariant probability measure (Proposition \ref{cph}).

In Section \ref{tracial}, we prove the non-commutative version of Proposition \ref{cph}. Namely, we show that a group is supramenable if and only if whenever it partially acts on a unital $\mathrm{C}^*$-algebra which has a tracial state, then the associated partial crossed product also has a tracial state (Theorem \ref{imp}). 

In \cite{gs}, Giordano and Sierakowski gave conditions which ensure that a partial crossed product is purely infinite. Our results, one the other hand, show that one cannot get a purely infinite $\mathrm{C}^*$-algebra out of a partial action of a supramenable group on a unital $\mathrm{C}^*$-algebra which has a tracial state.

Recall that a \emph{Kirchberg algebra} is a simple, nuclear, separable and purely infinite $\mathrm{C}^*$-algebra. Non-amenability of a group was used by Rørdam and Sierakowski in \cite{rs} for constructing unital Kirchberg algebras out of actions on the Cantor set. Analogously, non-supramenability of a group was used by Kellerhals, Monod and Rørdam in \cite{kmr} for constructing stable Kirchberg algebras out of actions on the locally compact, non-compact Cantor set. 

By using a result of \cite{kmr}, we show that if $G$ is a countable, amenable, non-supramenable group, then there exists a free, minimal, purely infinite and non-global partial action of $G$ on the Cantor set $K$  (Proposition \ref{kir}). It is a consequence of results from \cite{exel} and \cite{gs},  that the partial crossed product associated to any such partial action of $G$ on $K$ is a Kirchberg algebra.

Given a (global) action of a semi-direct product $G\rtimes H$ on a $\mathrm{C}^*$-algebra $A$, one can decompose the crossed product $A\rtimes(G\rtimes H)$ as two iterated crossed products by $G$ and $H$:
\begin{equation}
A\rtimes (G\rtimes H)\cong( A\rtimes G)\rtimes H.\label{isosd}
\end{equation}

 The class of supramenable groups is not closed under taking semi-direct products. For example, the \emph{lamplighter group} $\frac{\mathbb{Z}}{2\mathbb{Z}}\wr\mathbb{Z}$ contains a free monoid on two generators, hence cannot be supramenable. Therefore, our results show that, in general, one cannot have a decomposition such as in \eqref{isosd} for partial actions. 

In Section \ref{dec}, we give conditions under which \eqref{isosd} holds for partial actions. These conditions involve assuming that $A$ is stable or requiring that the domains of the partial isomorphisms satisfy a certain condition.
\begin{ack}
The author is grateful to his advisor, Mikael Rørdam, for his guidance throughout this work. The author would also like to thank Matias Andersen for very useful conversations on the topics of this work.
\end{ack}
\section{Supramenability and partial actions on compact Hausdorff spaces}\label{comm}
Throughout this article, we denote the identity element of a group by $e$ and all groups are assumed to be discrete.

We begin by recalling the definition of partial actions on topological spaces. See \cite{exel} for details and historical notes.

\begin{defi}
A \emph{partial action} $\theta$ of a group $G$ on a topological space $X$ is a pair $(\{D_g\}_{g\in G},\{\theta_g\}_{g\in G})$ where $\{D_g\}_{g\in G}$ is a family of open subsets of $X$ and each $\theta_g$ is a homeomorphism from $D_{g^{-1}}$ onto $D_g$ such that:
\begin{description}
\item{(i)} $\theta_e=\mathrm{Id}_X$;

\item{(ii)} For every $g,h\in G$ and $x\in D_{g^{-1}}$, if $\theta_g(x)\in D_{h^{-1}}$, then $x\in D_{(hg)^{-1}}$ and $\theta_{hg}(x)=\theta_h\circ\theta_g(x)$.
\end{description}

A consequence of the above definition is that, for every $g,h\in G$, 

\begin{equation}
\theta_g(D_{g^{-1}}\cap D_h)=D_g\cap D_{gh}.\label{doms}
\end{equation}

If $D_g=X$ for every $g\in G$, then the partial action is just a usual action of a group on a topological space. In this case, we might say that the partial action is \emph{global}.

In the definition of a partial action, it is possible that some of the open subsets $D_g$ are empty. For example, any group $G$ partially acts trivially on a topological space $X$ by letting $D_g:=\emptyset$ for every $g\neq e$ and $\theta_e:=\mathrm{Id}_X$.

\end{defi}
\begin{example}\label{restric}
Let $\theta$ be an action of a group $G$ on a topological space $X$. Given an open subset $D$ of $X$, let, for every $g\in G$, $D_g:=D\cap\theta_g(D)$. Then $$(\{D_g\}_{g\in G},\{\theta_g|_{D_{g^{-1}}}\}_{g\in G})$$ is a partial action of $G$ on $D$, called the \emph{restriction} of $\theta$ to $D$.
\end{example}

\begin{defi}
Let $(\{D_g\}_{g\in G},\{\theta_g\}_{g\in G})$ be a partial action of a group $G$ on a topological space $X$. We say a measure $\nu$ on $X$ is \emph{invariant} if, for all $E\in\mathcal{B}(X)$ and $g\in G$, we have that
$$\nu(\theta_g(E\cap D_{g^{-1}}))=\nu(E\cap D_{g^{-1}}).$$
\end{defi}

Next, we recall the definitions of supramenable groups and paradoxical subsets of a group. See \cite{s1993banach} for historical notes.

\begin{defi}
A group $G$ is \emph{supramenable} if for every non-empty subset $A$ of $G$ there is an invariant, finitely additive measure $\mu:\mathcal{P}(G)\to[0,+\infty]$ such that $\mu(A)=1$. 
\end{defi}

\begin{defi}\label{para}
Let $G$ be a group. We say a non-empty subset $A$ of $G$ is \emph{paradoxical} if there exist disjoint subsets $B$ and $C$ of $A$, finite partitions $\{B_i\}_{i=1}^n$ and $\{C_j\}_{j=1}^m$ of $B$ and $C$ and elements $s_1,...,s_n,t_1,...,t_m\in G$ such that $A=\sqcup_{i=1}^n s_iB_i=\sqcup_{j=1}^m t_jC_j$ (disjoint union).

\end{defi}
By Tarski's theorem, a non-empty subset $A$ of a group $G$ is not paradoxical if and only if there exists an invariant, finitely additive measure $\mu:\mathcal{P}(G)\to[0,+\infty]$ such that $\mu(A)=1$. Therefore, a group is supramenable if and only if it contains no paradoxical subsets.

In order to prove Proposition \ref{cph}, we will need the following lemma, whose proof can be found in  \cite[Proposition 2.1]{kmr}. 
\begin{lem}\label{func}
Let $G$ be a group and let $\mu$ be a finitely additive measure on $G$. Let $V_\mu$ be the the subspace of $\ell^\infty(G)$ consisting of all $f\in\ell^\infty(G)$ such that $\mu(\mathrm{supp} (f))<\infty$. It follows that there is a unique positive linear functional $I_\mu\colon V_\mu\to\mathbb{C}$ such that $I_\mu(1_E)=\mu(E)$ for all $E\subset G$ with $\mu(E)<\infty$. If $\mu$ is $G$-invariant, then so is $I_\mu$.
\end{lem}

\begin{prop}\label{cph}
A group is supramenable if and only if whenever it partially acts on a compact Hausdorff space, then the space admits an invariant probability measure. 

\end{prop}

\begin{proof}

Assume first that $G$ is a non-supramenable group. Then it has a paradoxical subset $A$. Let $j\colon G\to\beta G$ be the embedding of $G$ into its Stone-{\v{C}}ech compactification. Consider the partial action obtained by restricting the canonical action of $G$ on $\beta G$ to $\overline{j(A)}$. By using \cite[Lemma 2.4]{kmr}, one concludes that this partial action does not admit an invariant probability measure.

Conversely, let $(\{D_g\}_{g\in G},\{\theta_g\}_{g\in G})$ be a partial action of a supramenable group $G$ on a compact Hausdorff space $X$. Fix $x_0\in X$. Given $f\in C(X)$, let $\hat{f}\in \ell^\infty(G)$ be defined by 
$$\hat{f}(g):=\begin{cases}0 &\text{if } x_0\notin D_{g^{-1}} \\
f(\theta_g(x_0))  & \text{if } x_0\in D_{g^{-1}}. \end{cases}$$

By using \eqref{doms}, one can easily check that, for every $g\in G$ and $f\in C_0(D_{g^{-1}})$ (seen as an ideal of $C(X)$), it holds that 

\begin{equation}
g.\hat{f}=\widehat{f\circ\theta_{g^{-1}}}.\label{nvs}
\end{equation}

 Notice that, formally speaking, $f\circ\theta_{g^{-1}}$ is only defined on $D_g$. We see it as being defined on $X$ by extending it as $0$ outside of $D_g$.

Let $A:=\{g\in G:x_0\in D_{g^{-1}}\}$. Since $G$ is supramenable, there exists an invariant, finitely additive measure $\mu$ on $G$ such that $\mu(A)=1$.

In the notation of Lemma \ref{func}, let

\begin{align*}
\varphi\colon C(X)&\to\mathbb{C}\\
f&\mapsto I_\mu(\hat{f}).
\end{align*}

Notice that $\varphi$ is well-defined because, for every $f\in C(X)$, we have that $\mathrm{supp}(\hat{f})\subset A$.

Since $\widehat{1_{C(X)}}=1_A$, we obtain that $\varphi(1_{C(X)})=1$. Therefore, $\varphi$ is a state, since it is clearly positive. By \eqref{nvs} and the $G$-invariance of $I_\mu$, we conclude that, for every $g\in G$ and $f\in C_0(D_{g^{-1}})$, it holds that 
\begin{equation}
\varphi(f\circ\theta_{g^{-1}})=\varphi(f).\label{nvs2}
\end{equation}

Let $\nu$ be the regular probability measure on $X$ associated to $\varphi$ through the Riesz Representation theorem.  By \eqref{nvs2} and the inner regularity of $\nu$, we conclude that $\nu$ is invariant on open subsets. By outer regularity, it follows that $\nu$ is invariant for all Borel measurable sets.

\end{proof}

Next, we recall the definition and some facts about partial actions on $\mathrm{C}^*$-algebras and partial crossed products. See \cite{exel} for the details.
 
\begin{defi}\label{partiald}
A \emph{partial action} $\theta$ of a group $G$ on a $\mathrm{C}^*$-algebra $A$ is a pair $(\{I_g\}_{g\in G},\{\theta_g\}_{g\in G})$ where $\{I_g\}_{g\in G}$ is a family of closed two-sided ideals of $A$ and each $\theta_g$ is a $*$-isomorphisms from $I_{g^{-1}}$ onto $I_g$ such that:
\begin{description}
\item{(i)} $\theta_e=\mathrm{Id}_A$;

\item{(ii)} For every $g,h\in G$ and $x\in I_{g^{-1}}$, if $\theta_g(x)\in I_{h^{-1}}$, then $x\in I_{(hg)^{-1}}$ and $\theta_{hg}(x)=\theta_h\circ\theta_g(x)$.
\end{description}
The quadruple $(A, G, \{I_g\}_{g\in G},\{\theta_g\}_{g\in G})$ is called a \emph{partial dynamical system}.
\end{defi}

Given a partial dynamical system $$(A, G, \{I_g\}_{g\in G},\{\theta_g\}_{g\in G}),$$ denote by $C_c(G,A)$ the vector space of finitely supported functions from $G$ into $A$. Let 
$$A\rtimes_{\theta,\mathrm{alg}}G:=\{f\in C_c(G,A):f(g)\in I_g \text{ for every $g\in G$}\}.$$

For every $g\in G$ and $a_g\in I_g$, let $a_g\delta_g\in A\rtimes_{\theta,\mathrm{alg}}G$ be defined by
\begin{align*}
a_g\delta_g(h):=\begin{cases}a_g &\text{if } h=g \\
0 & \text{if } g\neq h \end{cases},\quad h\in G.
\end{align*}

Notice that $\{a_g\delta_g:g\in G,a_g\in I_g\}$ spans $A\rtimes_{\theta,\mathrm{alg}}G$. Hence we can define a product and an involution on $A\rtimes_{\theta,\mathrm{alg}}G$ by
\begin{align*}
(a_g\delta_g)(b_h\delta_h):=\theta_g(\theta_{g^{-1}}(a_g)b_h)\delta_{gh}\\
(a_g\delta_g)^*:=\theta_{g^{-1}}(a_g^*)\delta_{g^{-1}}
\end{align*}
for $g,h\in G$, $a_g\in I_g$ and $b_h\in I_h$. These operations turn $A\rtimes_{\theta,\mathrm{alg}}G$ into a $*$-algebra.

Define a seminorm on $A\rtimes_{\theta,\mathrm{alg}}G$ by
$$\|x\|_{\mathrm{max}}:=\sup\{p(x):\text{$p$ is a $\mathrm{C}^*$-seminorm on $A\rtimes_{\theta,\mathrm{alg}}G$}\}.$$

The \emph{partial crossed product} associated to the partial action $\theta$, denoted by $A\rtimes_\theta G$ or $A\rtimes G$, is the enveloping $\mathrm{C}^*$-algebra of the $*$-algebra $A\rtimes_{\theta,\mathrm{alg}}G$ endowed with the seminorm $\|\cdot\|_{\mathrm{max}}$.

There is an embedding

\begin{align*}
i\colon A&\to A\rtimes_\theta G\\
a&\mapsto a\delta_e
\end{align*}

and a conditional expectation $\psi\colon A\rtimes_\theta G\to A$ such that 
\begin{align}\label{psi}
\psi(a\delta_g)=\begin{cases}a &\text{if } g=e \\
0 & \text{if } g\neq e \end{cases},\quad g\in G, a\in I_g.
\end{align}

Notice that the construction of the partial crossed product is a generalization of the usual crossed product associated to an action of a group on a $\mathrm{C}^*$-algebra.

Given a partial action $\theta=(\{D_g\}_{g\in G},\{\theta_g\}_{g\in G})$ of a group $G$ on a locally compact Hausdorff space $X$, one can associate to it a partial action on $C_0(X)$, with domains $C_0(D_g)$ (seen as ideals of $C_0(X)$) and $*$-isomorphisms given by composition with the homeomorphisms $\theta_g$. 
\begin{defi}

Let $\theta=(\{D_g\}_{g\in G},\{\theta_g\}_{g\in G})$ be a partial action of a group $G$ on a Hausdorff space $X$. We say $\theta$ is \emph{free} if, for every $g\in G$ and $x\in D_{g^{-1}}$, $$\theta_g(x)=x \implies g=e.$$

The partial action is said to be \emph{minimal} if, given a closed set $F\subset X$ such that, for every $g\in G,\theta_g(F\cap D_{g^{-1}})\subset F$, we necessarily have that $F=\emptyset$ or $F=X$.

That is, the partial action is minimal if $X$ has no non-trivial invariant closed subsets.

If $X$ is totally disconnected, then the partial action is said to be \emph{purely infinite} if, for every compact-open subset $K$ of $X$, there exist pairwise disjoint compact-open subsets $K_1,...,K_{n+m}$ and elements $t_1,...,t_{n+m}\in G$ such that $K_j\subset K\cap D_{t_j^{-1}}$ for all $j$ and
$$K=\bigcup_{j=1}^n\theta_{t_j}(K_j)=\bigcup_{j=n+1}^{n+m}\theta_{t_j}(K_j).$$
\end{defi}

This definition of purely infinite partial action generalizes the one in \cite[Definition 4.4]{kmr}. Also, if the partial action is purely infinite, then every compact-open subset of $X$ is $(G,\tau_X)$-paradoxical, in the sense of \cite[Definition 4.3]{gs}.
\begin{prop}\label{kir}
Let $G$ be an amenable, non-supramenable, countable group. Then $G$ admits a free, minimal, purely infinite and non-global partial action on the Cantor set $K$. The partial crossed product associated to any such partial action of $G$ on $K$ is a Kirchberg algebra.
\end{prop}
\begin{proof}
Let $G$ be as in the statement. By \cite[Theorem 1.2]{kmr}, there is a free, minimal, purely infinite global action $\theta$ of $G$ on the non-compact, locally compact Cantor space $K^*$. The restriction of $\theta$ to a compact-open subset $K$ (which is automatically homeomorphic to the Cantor set) is clearly free and purely infinite. It is not global because it clearly does not admit an invariant probability measure and $G$ is amenable. Let us show that it is minimal.

Suppose that there exists a non-empty closed set $F\subsetneq K$ such that, for every $g\in G$, 
\begin{align}\label{cont}
\theta_g(F\cap\theta_{g^{-1}}(K))\subset F.
\end{align}
 
By minimality of $\theta$, the orbit of evey point is dense. Hence, given $x\in F$, there is $g\in G$ such that $\theta_g(x)$ belongs to the non-empty open set $K\setminus F$, but, since $x \in F\cap\theta_{g^{-1}}(K)$, this contradicts \eqref{cont}.

The fact that the associated partial crossed product is simple follows from \cite[Corollary 29.8]{exel}. The fact that it is nuclear is due to \cite[Proposition 25.10]{exel}. It is purely infinite because of \cite[Theorem 4.4]{gs}.
\end{proof}

\section{Tracial states on partial crossed products}\label{tracial}

In this section, we prove the main result of this article (Theorem \ref{imp}). We start by proving some lemmas about traces in $\mathrm{C}^*$-algebras.

\begin{lem}\label{trace}
Let $A$ be a $\mathrm{C}^*$-algebra, $I$ an ideal of $A$ and $\tau$ a positive linear functional on $I$. Then there is a unique positive linear functional $\tau'$ on $A$ such that $\tau'$ extends $\tau$ and $\|\tau\|=\|\tau'\|$. Moreover, if $\tau$ is a trace, then $\tau'$ is also a trace.
\end{lem}
\begin{proof}
The existence and uniqueness of $\tau'$ is the content of \cite[Theorem 3.3.9]{murphy}. Assume that $\tau$ is a trace and we will show that $\tau'$ is also a trace. 

By taking the unitization, we can assume $A$ is unital.

Let $(u_\lambda)_{\lambda\in\Lambda}$ be an approximate unit for $I$. Notice that $$\tau'(1)=\|\tau'\|=\|\tau\|=\lim_\lambda\tau(u_\lambda).$$

Now suppose that $(x_\lambda)_{\lambda\in\Lambda}$ is a bounded net in $A$ such that $\lim_\lambda\tau'(x_\lambda)$ exists. Let us show that $\lim_\lambda\tau'(x_\lambda)=\lim_\lambda\tau(x_\lambda u_\lambda).$
 Indeed,
\begin{align*}
|\tau'(x_\lambda)-\tau(x_\lambda u_\lambda)|=|\tau'(x_\lambda(1-u_\lambda))|&\stackrel{(*)}\leq\tau'(x_\lambda x_\lambda^*)^{\frac{1}{2}}\tau'((1-u_\lambda)^2)^{\frac{1}{2}}\\
&\stackrel{(**)}\leq\|\tau'\|^\frac{1}{2}\|x_\lambda\|\tau'(1-u_\lambda)^\frac{1}{2}\to 0.
\end{align*}
The inequality in (*) is due to the Cauchy-Schwarz inequality and the one in (**) is due to the fact that, for every $\lambda\in \Lambda$, it holds that $(1-u_\lambda)^2\leq 1-u_\lambda$.

Let $a,b\in A$ and let us show that $\tau'(ab)=\tau'(ba).$
\begin{align*}
\tau'(ab)&=\lim_\lambda\tau(abu_\lambda)=\lim_\lambda\tau(u_\lambda^\frac{1}{2}abu_\lambda^\frac{1}{2})=\lim_\lambda\tau(bu_\lambda a)\\
&=\lim_\lambda\tau(bu_\lambda a u_\lambda)=\lim_\lambda\tau(au_\lambda bu_\lambda)=\tau'(ba).
\end{align*}
\end{proof}
\begin{definition}
Let $A$ be a $\mathrm{C}^*$-algebra, $G$ a group and 
$(A, G, \{I_g\}_{g\in G},\{\theta_g\}_{g\in G})$ a partial dynamical system. We say a linear functional $\tau$ on $A$ is \emph{$G$-invariant} if for every $g\in G$ and $a\in I_{g}$
$$\tau(a)=\tau(\theta_{g^{-1}}(a)).$$
\end{definition}

\begin{lem}\label{inv}
Let $A$ be a $\mathrm{C}^*$-algebra, $G$ a group and 
$(A, G, \{I_g\}_{g\in G},\{\theta_g\}_{g\in G})$ a partial dynamical system.  Let $\psi$ be the canonical conditional expectation from $A\rtimes G$ onto $A$ as in \eqref{psi}. If $\varphi$ is a tracial state on $A$, then $\varphi\circ\psi$ is a tracial state on $A\rtimes_\theta G$ if and only if $\varphi$ is a $G$-invariant. 
\end{lem}
\begin{proof}
Assume $\varphi$ is $G$-invariant. Let us show that $\varphi\circ\psi$ is a tracial state. By the fact that $\psi$ is positive and by \eqref{psi}, it follows that $\varphi\circ\psi$ is a state.  

By linearity and continuity, in order to conclude that $\varphi\circ\psi$ is tracial, we only need to prove that, given $a_g\delta_g, b_h\delta_h \in A\rtimes_\theta G$,
$$\varphi\circ\psi((a_g\delta_g)(b_h\delta_h))=\varphi\circ\psi((b_h\delta_h)(a_g\delta_g)).$$

This is true, since
$$\varphi\circ\psi(a_g\delta_g b_h\delta_h)=\varphi\circ\psi(\theta_g(\theta_{g^{-1}}(a_g)b_h)\delta_{gh})=\begin{cases}0 &\text{if } h\neq g^{-1} \\
 \varphi(a_g\theta_g(b_{g^{-1}})) & \text{if } h=g^{-1}. \end{cases}$$

By using the fact that $\varphi$ is a $G$-invariant tracial state, one concludes that $\varphi\circ\psi(a_g\delta_g b_h\delta_h)=\varphi\circ\psi(b_h\delta_h a_g\delta_g)$. Hence, $\varphi\circ\psi$ is a tracial state.

Conversely, assume that $\varphi\circ\psi$ is a tracial state and let us show that $\varphi$ is $G$-invariant.

Given $g\in G$, $a\in I_g$ and $(u_\lambda)_{\lambda\in \Lambda}$ an approximate unit for $I_{g^{-1}}$, we have that
\begin{align*}
\varphi(a)=\varphi\circ\psi(a\delta_e)=\lim_\lambda\varphi\circ\psi(a\delta_gu_\lambda\delta_{g^{-1}})
=\lim_\lambda\varphi\circ\psi(u_\lambda\delta_{g^{-1}}a\delta_g)\\
=\varphi\circ\psi(\theta_{g^{-1}}(a)\delta_e)=\varphi(\theta_{g^{-1}}(a)).  
\end{align*}

Hence, $\varphi$ is $G$-invariant.
\end{proof}

Just as in the case of global actions, we have that for partial actions the following holds:
\begin{prop}\label{med}
Let $\theta$ be a partial action of a group $G$ on a compact Hausdorff space $X$. Then $X$ admits an invariant probability measure if and only if $C(X)\rtimes_\theta G$ has a tracial state.
\end{prop}
\begin{proof}
Supose $X$ has an invariant probability measure $\mu$. Consider the state on $C(X)$ given by integration with respect to $\mu$. By invariance of $\mu$, it follows that the state is also $G$-invariant. By Lemma \ref{inv}, we conclude that $C(X)\rtimes_\theta G$ has a tracial state.

Now assume that $C(X)\rtimes_\theta G$ has a tracial state $\tau$. Let $i\colon C(X)\to C(X)\rtimes_\theta G$ be the canonical embedding and $\psi\colon C(X)\rtimes_\theta G\to C(X)$ the canonical conditional expectation. We claim that $\tau\circ i$ is $G$-invariant. By Lemma $\ref{inv}$, we only need to check that $\tau\circ i\circ\psi$ is a tracial state. Clearly, it is a state.

Given $a_g\delta_g, b_h\delta_h\in C(X)\rtimes G$, we have that $\psi(a_g\delta_gb_h\delta_h)=\psi(b_h\delta_ha_g\delta_g)=0$ if $h\neq g^{-1}$. If $h=g^{-1}$, then  $\tau\circ i\circ\psi(a_g\delta_gb_h\delta_h)=\tau(a_g\delta_gb_h\delta_h)
=\tau(b_h\delta_ha_g\delta_g)=\tau\circ i\circ\psi(b_h\delta_ha_g\delta_g)$.

By linearity, we conclude that $\tau\circ i\circ\psi$ is a tracial state. Therefore, $\tau\circ i$ is a $G$-invariant state. Let $\mu$ be the regular probability measure on $X$ associated to $\tau\circ i$. Since this state is $G$-invariant, we conlude, just as in the proof of Proposition \ref{cph}, that $\mu$ is invariant.
\end{proof}

In the following, we denote by $T(A)$ the set of tracial states of a C$^*$-algebra $A$.
\begin{lem}\label{pure}
Let $A$ be a unital $\mathrm{C}^*$-algebra such that $T(A)\neq\emptyset$. Then every extreme point $\tau$ of $T(A)$ satisfies that for every ideal $I$ of $A$, $\|\tau\big|_I\|$ is either $0$ or $1$.
\end{lem}
\begin{proof}

Let $\tau$ be an extreme point of $T(A)$. Given an ideal $I\unlhd A$, suppose that $t:=\|\tau\big|_I\|\in(0,1)$. By Lemma \ref{trace}, there exists an extension $\tau_I$ of $\tau\big|_I$ to all of $A$ such that $\|\tau_I\|=t$.

Let us show that $\tau_I\leq\tau$. Let $(u_\lambda)_{\lambda\in \Lambda}$ an approximate unit for $I$. Given $a\in A^+$, we know by \cite[Theorem 3.3.9]{murphy} that $\tau_I(a)=\lim_\lambda\tau(u_\lambda au_\lambda)$. Since $\tau$ is a tracial state, we have that, for any $\lambda\in \Lambda$, 
\begin{align*}
\tau(u_\lambda au_\lambda)=\tau(a^\frac{1}{2}u_\lambda u_\lambda a^\frac{1}{2})
\leq\|(u_\lambda)^2\|\tau(a)\leq\tau(a).
\end{align*}
Therefore, $\tau_I(a)\leq\tau(a)$. Since $a$ was arbitrary, it follows that $\tau_I\leq\tau$.

By evaluating on $1_A$, we get that $\|\tau-\tau_I\|=1-t$. Hence,
$$\tau=t\left(\frac{\tau_I}{\|\tau_I\|}\right)+(1-t)\left(\frac{\tau-\tau_I}{\|\tau-\tau_I\|}\right).$$
Since $\tau$ is an extreme point, we get that $\tau=\frac{\tau_I}{\|\tau_I\|}$. But $\tau\big|_I=\tau_I\big|_I\neq 0$. Thus we get a contradiction. 

\end{proof}

The reader will notice that the idea of the proof of the next result is the same of Proposition \ref{cph}
\begin{theorem}\label{imp}
A group is supramenable if and only if whenever it partially acts on a unital $\mathrm{C}^*$-algebra which has a tracial state, then the associated partial crossed product also has a tracial state.
\end{theorem}
\begin{proof}
Let $G$ be a supramenable group, $A$ a unital $\mathrm{C}^*$-algebra which has a tracial state and $(A, G, \{I_g\}_{g\in G},\{\theta_g\}_{g\in G})$ a partial dynamical system.

By Lemma \ref{pure}, $A$ has a tracial state $\tau$ such that for every ideal $I$ of $A$, 
$\|\tau\big|_I\|$ is either $0$ or $1$.

For each $g\in G$, let $\tau_g$ be the extension of $\tau\circ\theta_{g^{-1}}$ to $A$, as in Lemma \ref{trace}. For each $a\in A$, let $\hat{a}\in \ell^\infty(G)$ be defined by $\hat{a}(g):=\tau_g(a)$, $g\in G$. Since each $\tau_g$ is a trace, we have that, for any $a,b\in A$, $\widehat{ab}=\widehat{ba}$.

Notice that for each $g\in G$, $\widehat{1_A}(g)=\|\tau_g\|=\|\tau\big|_{I_{g^{-1}}}\|$.

Since $G$ is supramenable, there is an invariant, finitely addivite measure $\mu$ on $G$ such that $1=\mu(\{g\in G:\|\tau\big|_{I_{g^{-1}}}\|=1\}).$  

In the notation of Lemma \ref{func}, define
\begin{align*}
\varphi\colon A&\to\mathbb{C}\\
a&\mapsto I_\mu(\hat{a}).
\end{align*}
Notice that $\varphi$ is a tracial state on $A$. 

Let $\psi$ be the canonical conditional expectation from $A\rtimes G$ onto $A$ as in \eqref{psi}.
Define $\rho:=\varphi\circ\psi$. By Lemma \ref{inv}, in order to show that $\rho$ is a tracial state, it is sufficient to show that $\varphi$ is $G$-invariant.

Given $x,y\in G$, let us prove that
$$\|\tau_x\big|_{I_y}\|=\|\tau_x\big|_{I_y\cap I_{x}}\|.$$

Let $(u_\lambda^y)$ and $(u_\lambda^x)$ be approximate 
units for $I_y$ and $I_x$, respectively.
Then
$$\|\tau_x\big|_{I_y}\|=\lim_{\lambda_1}\tau_x(u_{\lambda_1}^y)=
\lim_{\lambda_1}\lim_{\lambda_2}
\tau_x(u_{\lambda_1}^y u_{\lambda_2}^x).$$

Since, for each $\lambda_1, \lambda_2$, we have that $u_{\lambda_1}^y u_{\lambda_2}^x\in I_x\cap I_y$,
it follows that $\|\tau_x\big|_{I_y}\|\leq\|\tau_x\big|_{I_y\cap I_{x}}\|$. The opposite 
inequality is evident.

For each $g,t\in G$, one can easily check that 
$(\tau_t\circ\theta_{g^{-1}})\big|_{I_g\cap I_{gt}}=\tau_{gt}\big|_{I_g\cap I_{gt}}.$
Therefore,
\begin{align*}
\|\tau_t\circ\theta_{g^{-1}}\|&=\|\tau_t\big|_{I_{g^{-1}}}\|=\|\tau_t\big|_{I_{g^{-1}}\cap I_t}\|
=\|\tau_t\circ\theta_{g^{-1}}\circ\theta_g\big|_{I_{g^{-1}}\cap I_t}\|\\
&=\|\tau_t\circ\theta_{g^{-1}}\big|_{I_g\cap I_{gt}}\|
=\|\tau_{gt}\big|_{I_g\cap I_{gt}}\|=\|\tau_{gt}\big|_{I_g}\|.
\end{align*}
It follows by Lemma \ref{trace} that $\tau_t\circ\theta_{g^{-1}}=\tau_{gt}\big|_{I_g}$. Hence, given $a\in I_g$, 
$$g^{-1}.\hat{a}(t)=\hat{a}(gt)=\widehat{\theta_{g^{-1}}(a)}(t).$$
Therefore, 
\begin{align*}g^{-1}.\hat{a}=\widehat{\theta_{g^{-1}}(a)}.
\end{align*}
Thus, for every $g\in G$ and $a\in I_g$,
\begin{align*}
\varphi(a)=I_\mu(\hat{a})\stackrel{(*)}=I_\mu(g^{-1}\hat{a})
=I_\mu(\widehat{\theta_{g^{-1}}(a)})=\varphi(\theta_{g^{-1}}(a)).
\end{align*}
The equality (*) is due to the fact that $\mu$ is invariant.
Hence, $\varphi$ is $G$-invariant and $\rho$ is a tracial state on $A\rtimes G$.

The converse follows from Propositions \ref{med} and \ref{cph}.
\end{proof}

\section{Decomposition of partial crossed products} \label{dec}

In this section, we investigate conditions under which one can decompose a partial crossed product as two iterated partial crossed products by simpler groups.

Let $A$ be a $\mathrm{C}^*$-algebra, $G$ a group and $(A,G,\{I_g\}_{g\in G},\{\theta_g\}_{g\in G})$ a partial dynamical system. Given a $G$-invariant ideal $I$ of $A$ (i.e., $\theta_g(I\cap I_{g^{-1}})\subset I$ for every $g\in G$), one can consider the restricted partial dynamical system $(I,G,\{I\cap I_{g}\}_{g\in G},\{\theta_g\big|_{I\cap I_{g^{-1}}}\}_{g\in G})$. It is a consequence of \cite[Theorem 22.9]{exel} that the map
\begin{align*}
i\colon I\rtimes G&\to A\rtimes G\\
a\delta_g&\mapsto a\delta_g
\end{align*}
is an embedding onto an ideal of $A\rtimes G$. In the statement of the following lemma, we use this identification.
\begin{lem}
Let $(A,G,\{I_g\}_{g\in G},\{\theta_g\}_{g\in G})$ be a partial dynamical system and $I$ and $J$ be $G$-invariant ideals of $A$. Then $(I\rtimes G)\cap (J\rtimes G)=(I\cap J)\rtimes G$.
\end{lem}
\begin{proof}
Take $x\in(I\rtimes G)\cap (J\rtimes G)$. Let $(u_\lambda)_{\lambda\in\Lambda}$ be an approximate unit for $J$. Then $x=\lim_\lambda x(u_\lambda\delta_e)$. Since $x\in I\rtimes G$, it follows that $x$ can be approximated by finite sums $\sum a_g\delta_g$ such that each $a_g\in I$. Due to $G$-invariance of $I$ and $J$, if $a_g\in I$, then $a_g\delta_gu_\lambda\delta_e\in (I\cap J)\rtimes G$. Hence, for every $\lambda\in \Lambda$, $x(u_\lambda\delta_e)\in (I\cap J)\rtimes G$. From this, we conclude that $x\in(I\cap J)\rtimes G$.

The opposite inclusion is evident. 
\end{proof}
\begin{prop}\label{domains}
Let A be a C*-algebra, $G\rtimes_\alpha H$ a semi-direct product of groups and 
$$(A,G\rtimes_\alpha H,\{I_{(g,h)}\}_{(g,h)\in G\rtimes_\alpha H},\{\theta_{(g,h)}\}_{(g,h)\in G\rtimes_\alpha H})$$ a partial dynamical system such that for every $g\in G$ and $h\in H$,
\begin{equation}\label{cond}
I_{(g,h)}\subset I_{(g,e)}\cap I_{(e,h)}.
\end{equation}

Consider the partial dynamical system $$(A,G,\{I_{(g,e)}\}_{g\in G},\{\theta_{(g,e)}\}_{g\in G})$$ obtained by restricting the partial action $\theta$ to the subgroup $G\times\{e\}$. There is a partial action of $H$ on $A\rtimes G$ and a $*$-isomorphism
\begin{align*}
\varphi\colon A\rtimes({G\rtimes_\alpha H})&\to(A\rtimes G)\rtimes H\\
a_{(g,h)}\delta_{(g,h)}&\mapsto (a_{(g,h)}\delta_g)\delta_h.
\end{align*}
\end{prop}
\begin{proof}
Notice that condition \eqref{cond} and \cite[Proposition 2.6]{exel} imply that for each $h\in H$, $I_{(e,h)}$ is a $G$-invariant ideal for the partial dynamical system
$$(A,G,\{I_{(g,e)}\}_{g\in G},\{\theta_{(g,e)}\}_{g\in G}).$$

Also, for each $g\in G$ and $h\in H$, it is true that $\theta_{(e,h)}(I_{(e,h^{-1})}\cap I_{(g,e)})\subset I_{(\alpha_h(g),e)}$. For every $h\in H$, define
\begin{align*}
\beta_h\colon I_{(e,h^{-1})}&\rtimes G\to I_{(e,h)}\rtimes G\\
a_g\delta_g&\mapsto\theta_{(e,h)}(a_g)\delta_{\alpha_h(g)}.
\end{align*}

Let us check that $\beta_h$ is a well-defined $*$-homomorphism. Given $g,l\in G$, $a_g\in I_{(e,h^{-1})}\cap I_{(g,e)}$ and $a_l\in I_{(e,h^{-1})}\cap I_{(l,e)}$, we have that 
\begin{align*}
\beta_h((a_g\delta_g)(a_l\delta_l))&=\beta_h(\theta_{(g,e)}(\theta_{(g^{-1},e)}(a_g)a_l)\delta_{gl})\\
&=\theta_{(e,h)}(\theta_{(g,e)}(\theta_{(g^{-1},e)}(a_g)a_l))\delta_{\alpha_h(gl)}\\
&=\theta_{(\alpha_h(g),e)}(\theta_{(e,h)}(\theta_{(g^{-1},e)}(a_g)a_l))\delta_{\alpha_h(gl)}\\
&=\theta_{(\alpha_h(g),e)}(\theta_{(\alpha_h(g^{-1}),e)}(\theta_{(e,h)}(a_g))\theta_{(e,h)}(a_l))\delta_{\alpha_h(gl)}\\
&=(\theta_{(e,h)}(a_g)\delta_{\alpha_h(g)})(\theta_{(e,h)}(a_l)\delta_{\alpha_h(l)})\\
&=\beta_h(a_g\delta_g)\beta_h(a_l\delta_l).
\end{align*}

We leave it to the reader to check that, for every $g\in G$ and $a_g\in I_{(e,h^{-1})}\cap I_{(g,e)}$, it holds that $\beta_h((a_g\delta_g)^*)=\beta_h(a_g\delta_g)^*$. 
Therefore, $\beta_h$ is a well-defined $*$-homomorphism. Also $\beta_{h^{-1}}=(\beta_h)^{-1}$ for every $h\in H$.

 We claim that
$$(A\rtimes G,H,\{I_{(e,h)}\rtimes G\}_{h\in H},\{\beta_h\}_{h\in H})$$ is a partial dynamical system. 

Obviously, condition (i) of Definition \ref{partiald} holds, so we are left with checking that (ii) also holds.

Given $h_1,h_2\in H$ and $x\in I_{(e,h_1^{-1})}\rtimes G$, suppose that $\beta_{h_1}(x)\in I_{(e,h_2^{-1})}\rtimes G$ and let us show that $x\in I_{(e,(h_2h_1)^{-1})}\rtimes G$. 

Since the image of $\beta_{h_1}$ is $I_{(e,h_1)}\rtimes G$, we get that $$\beta_{h_1}(x)\in (I_{(e,h_2^{-1})}\rtimes G)\cap( I_{(e,h_1)}\rtimes G)=(I_{(e,h_2^{-1})}\cap I_{(e,h_1)})\rtimes G.$$

Since $\theta_{(e,h_1^{-1})}(I_{(e,h_2^{-1})}\cap I_{(e,h_1)})\subset I_{(e,(h_2h_1)^{-1})}$, we conclude, by using the definition of $\beta_{h_1^{-1}}$, that
$$
\beta_{h_1^{-1}}((I_{(e,h_2^{-1})}\cap I_{(e,h_1)})\rtimes G) \subset I_{(e,(h_2h_1)^{-1})}\rtimes G.
$$
Hence, $x\in I_{(e,(h_2h_1)^{-1})}\rtimes G$.

Now let us show that 
\begin{equation}
\beta_{h_2h_1}(x)=\beta_{h_2}\circ\beta_{h_1}(x). \label{beta}
\end{equation}
 Since $x\in (I_{(e,(h_2h_1)^{-1})}\cap I_{(e,h_1^{-1})})\rtimes G$, by using approximation arguments we can assume that $x=a\delta_g$ for some $g\in G$ and $a\in I_{(e,(h_2h_1)^{-1})}\cap I_{(e,h_1^{-1})}$. Since $\theta_{(e,h_2h_1)}(a)=\theta_{(e,h_2)}\circ \theta_{(e,h_1)}(a)$, we conclude that \eqref{beta} holds.

Hence, we have proven that
$$(A\rtimes G,H,\{I_{(e,h)}\rtimes G\}_{h\in H},\{\beta_h\}_{h\in H})$$ is a partial dynamical system. 

Define
\begin{align*}
\varphi\colon A\rtimes({G\rtimes_\alpha H})&\to(A\rtimes G)\rtimes H\\
a_{(g,h)}\delta_{(g,h)}&\mapsto (a_{(g,h)}\delta_g)\delta_h.
\end{align*}
For every $h\in H$, define

\begin{align*}
\psi_h\colon I_{(e,h)}\rtimes G&\to A\rtimes(G\rtimes H)\\
a_g\delta_g&\mapsto a_g\delta_{(g,e)}.
\end{align*}

It is a straightforward computation to check that $\varphi$ and, for every $h\in H$, $\psi_h$ are well-defined $*$-homomorphisms.
 
For every $h\in H$, let $(u_\lambda^h)_{\lambda\in\Lambda_h}$ be an approximate unit for $I_{(e,h)}$. Given $g\in G$, $h\in H$ and $a\in I_{(e,h)}\cap I_{(g,e)}$, notice that
\begin{align}\label{conv}
\lim_\lambda(a\delta_{(g,e)})(u_\lambda^h\delta_{(e,h)})=\lim_\lambda\theta_{(g,e)}(\theta_{(g^{-1},e)}(a)u_\lambda^h)\delta_{(g,h)}=a\delta_{(g,h)},
\end{align}
since $\theta_{(g^{-1},e)}(a)\in  I_{(g^{-1},h)}\subset  I_{(e,h)}$.

Define
\begin{align*}
\psi\colon (A\rtimes G)\rtimes H&\to A\rtimes(G\rtimes H)\\
f_h\delta_h&\mapsto\lim_\lambda(\psi_h(f_h)(u_\lambda^h\delta_{(e,h)})).
\end{align*}

In order to see that the above limit is well-defined, notice that, due to \eqref{conv}, convergence is assured for $f_h=a_g\delta_g$. Then, one observes that $(\psi_h(f)(u_\lambda^h\delta_{(e,h)}))$ is a Cauchy net for any $f\in I_{(e,h)}\rtimes G$.

For the proof that $\psi$ is a well-defined $*$-homomorphism, it is worthwhile remarking that, for every $h\in H$, $\psi\big|_{(I_{(e,h)}\rtimes G)\delta_h}$ is a contractive linear map.

Notice that, given $g\in G$, $h\in H$ and $a\in I_{(e,h)}\cap I_{(g,e)}$, it holds that $\psi((a\delta_g))\delta_h)=a\delta_{(g,h)}$. Hence, $\psi$ and $\varphi$ are inverses of each other.

\end{proof}

The purpose of the next proposition is to show that condition \eqref{cond} is stronger than it seems at first sight. We give a direct proof of it, even though it can also be obtained as a corollary of the proof of the preceding proposition.
\begin{prop}
Let A be a $\mathrm{C}^*$-algebra, $G\rtimes_\alpha H$ a semi-direct product of groups and 
$$(A,G\rtimes_\alpha H,\{I_{(g,h)}\}_{(g,h)\in G\rtimes_\alpha H},\{\theta_{(g,h)}\}_{(g,h)\in G\rtimes_\alpha H})$$ a partial dynamical system such that, for every $g\in G$ and $h\in H$, \eqref{cond} is satisfied.
Then, for every $g\in G$ and $h\in H$,
\begin{equation*}
I_{(g,h)}= I_{(g,e)}\cap I_{(e,h)}.
\end{equation*}
\end{prop}
\begin{proof}
Given $g\in G$ and $h\in H$,
\begin{align*}
 I_{(g,e)}\cap I_{(e,h)}&=\theta_{(g,e)}\circ\theta_{(g^{-1},e)}(I_{(g,e)}\cap I_{(e,h)})\\
&=\theta_{(g,e)}(I_{(g^{-1},e)}\cap I_{(g^{-1},h)})\\
&\subset\theta_{(g,e)}(I_{(g^{-1},e)}\cap I_{(e,h)})\\
&= I_{(g,e)}\cap I_{(g,h)}\\
&= I_{(g,h)}.
\end{align*}
\end{proof}

Next, we give an example where Proposition \ref{domains} can be applied.
\begin{example}
Let $G$ and $H$ be groups and $j:G\times H\to \beta(G\times H)$ be the embedding of $G\times H$ into its Stone-{\v{C}}ech compactification. Given $A\subset G$ and $B\subset H$ non-empty subsets, consider the partial action $(\{D_g\}_{g\in G},\{\theta_g\}_{g\in G})$ of $G\times H$ on $\overline{j(A\times B)}$ obtained by restricting to $\overline{j(A\times B)}$ the canonical action of  $G\times H$ on $\beta(G\times H)$, as in Example \ref{restric}. By definition of the restriction, given $(g,h)\in G\times H$, we have that
\begin{align*}
D_{(g,h)}=\overline{j((A\cap gA)\times (B\cap hB))}&=\overline{j(A\times (B\cap hB))}\cap
\overline{j((A\cap gA)\times B)}\\
&=D_{(g,e)}\cap D_{(e,h)}.
\end{align*}

From this, it follows easily that the induced partial action on $C(\overline{j(A\times B)})$ satisfies \eqref{cond}. 
\end{example}
  
Notice that, if $G$ and $H$ are supramenable groups, we can combine the previous example with Proposition \ref{domains} and Theorem \ref{imp} to show that $A\times B\subset G\times H$ is non-paradoxical for every $A\subset G$ and $B\subset H$ non-empty subsets. We remark, however, that this fact also follows easily from results of \cite{rosenblatt1973generalization}.

As mentioned in the introduction, the class of supramenable groups is not closed under taking semi-direct products. For example, let $(x_n)_{n\in\mathbb{Z}}\in\bigoplus_\mathbb{Z}\frac{\mathbb{Z}}{2\mathbb{Z}}$ be such that $x_0=1$ and $x_n=0$ for $n\neq 0$. Also let $e\in\bigoplus_\mathbb{Z}\frac{\mathbb{Z}}{2\mathbb{Z}}$ be the neutral element. Now consider the lamplighter group $\frac{2\mathbb{Z}}{\mathbb{Z}}\wr\mathbb{Z}=(\bigoplus_\mathbb{Z}\frac{\mathbb{Z}}{2\mathbb{Z}})\rtimes\mathbb{Z}$. It is well-known and easy to check that $((x_n)_{n\in\mathbb{Z}},1),(e,1)\in G$ generate a free monoid $SF_2$. Since $SF_2\subset \frac{2\mathbb{Z}}{\mathbb{Z}}\wr\mathbb{Z}$ is obviously paradoxical, it follows that $\frac{2\mathbb{Z}}{\mathbb{Z}}\wr\mathbb{Z}$ is not supramenable.

Let $G\rtimes H$ be a non-supramenable group which is the semi-direct product of supramenable groups. By Theorem \ref{imp}, there is a partial action of $G\rtimes H$ on a unital $\mathrm{C}^*$-algebra with a tracial state such that the associated partial crossed product does not admit a tracial state. By applying Theorem \ref{imp} twice, we conclude that a decomposition such as in Proposition \ref{domains} cannot hold in this case.

Next, we are going to provide another condition which allows a decomposition as in \eqref{isosd}. For the proof of it, we will use the concepts of Fell bundle, graded $\mathrm{C}^*$-algebra, cross sectional $\mathrm{C}^*$-algebra of a Fell bundle and reduced partial crossed product. We refer the reader to \cite{exel} for the appropriate definitions.

Given a partial action $\theta$ of a group $G$ on a $\mathrm{C}^*$-algebra $A$ and $H$ a subgroup of $G$, one can consider the partial action of $H$ on $A$ obtained by restricting $\theta$ to $H$. We shall denote this restricted partial action also by $\theta$.
\begin{lem}\label{condsd}
Let $(A, G, \{I_g\}_{g\in G},\{\theta_g\}_{g\in G})$ be a partial dynamical system and $H$ a subgroup of $G$. Then $A\rtimes_{\theta} H$ embeds naturally into $A\rtimes_\theta G$ and there exists a conditional expectation $\psi_H\colon A\rtimes_\theta G\to A\rtimes_{\theta} H$ such that, for every $g\in G$ and $a\in I_g$,
 \begin{align}\label{h}
\psi_H(a\delta_g)=\begin{cases}a\delta_g &\text{if } g\in H \\
0 & \text{if } g\notin H. \end{cases}
\end{align}
\end{lem}
\begin{proof}
The fact that $A\rtimes_\theta H$ embeds naturally into $A\rtimes_\theta G$ was proven in \cite[Corollary 6.3]{ara2013dynamical}.

Let us now prove the existence of $\psi_H$. By \cite[Proposition 6.1]{ara2013dynamical}, we have that the reduced partial crossed product $A\rtimes_{\theta,\mathrm{r}} H$ embeds naturally into $A\rtimes_{\theta,\mathrm{r}}G$ and there exists a conditional expectation $\varphi_H\colon A\rtimes_{\theta,\mathrm{r}}G\to A\rtimes_{\theta,\mathrm{r}} H$ satisfying the same condition as in \eqref{h}.

Let $\mathrm{C}^*_\mathrm{r}(G)$ and $\mathrm{C}^*_\mathrm{r}(H)$ denote the reduced group $\mathrm{C}^*$-algebras of $G$ and $H$ and $E_H\colon \mathrm{C}^*_\mathrm{r}(G)\to \mathrm{C}^*_\mathrm{r}(H)$ the conditional expectation satisfying the same condition as in \eqref{h}.

Let 
\begin{align*} 
i_G\colon A\rtimes_\theta G&\to (A\rtimes_{\theta,\mathrm{r}}G)\otimes_\mathrm{max}\mathrm{C}^*_\mathrm{r}(G)
\end{align*}

 and 
\begin{align*}i_H\colon A\rtimes_\theta H&\to (A\rtimes_{\theta,\mathrm{r}}H)\otimes_\mathrm{max}\mathrm{C}^*_\mathrm{r}(H) 
\end{align*}
be the injective $*$-homomorphisms given by \cite[Proposition 18.9]{exel} such that $i_G(a_g\delta_g)= a_g\delta_g\otimes\delta_g$ and $i_H(a_h\delta_h)=a_h\delta_h\otimes\delta_h$ for every $g\in G$, $a_g\in I_g$, $h\in H$ and $a_h\in I_h$.

Then $\psi_H:=(i_H)^{-1}\circ(\varphi_H\otimes E_H)\circ i_G$ is the desired conditional expectation.

\end{proof}

\begin{theorem}\label{stable}
Let $A$ be a separable and stable $\mathrm{C}^*$-algebra, $G$ a countable group, $N$ a normal subgroup of $G$ and $(A, G, \{I_g\}_{g\in G},\{\theta_g\}_{g\in G})$ a partial dynamical system. Then there is a partial action $\beta$ of $\frac{G}{N}$ on $A\rtimes_{\tilde{\theta}}N$ such that 
$$A\rtimes_\theta G\cong(A\rtimes_{\theta}N)\rtimes_\beta\frac{G}{N}.$$
\end{theorem}
\begin{proof}

For each $gN\in \frac{G}{N}$, define
$$B_{gN}:=\overline{\text{span}}\{b_h\delta_h:h\in gN,\ \ b_h\in I_h\}.$$ 

We want to show that $(A\rtimes_\theta G,\{B_{gN}\}_{gN\in\frac{G}{N}})$ is a $\frac{G}{N}$-graded $\mathrm{C}^*$-algebra.

Clearly, for every $g,g'\in G$, we have that $B_{gN}B_{g'N}\subset B_{gg'N}$ and $(B_{gN})^*=B_{g^{-1}N}$. Notice that, by Lemma \ref{condsd}, we have that $B_{eN}\cong A\rtimes_{\theta}N$ and there exists a conditional expectation $\psi_N\colon A\rtimes_\theta G\to B_{eN}$ with the same property as in $\eqref{h}$.

By \cite[Theorem 19.1]{exel}, it follows that $$(A\rtimes_\theta G,\{B_{gN}\}_{gN\in\frac{G}{N}})$$ is a $\frac{G}{N}$-graded $\mathrm{C}^*$-algebra. Let $\mathcal{B}$ be the Fell bundle associated to this grading. There is a surjective $*$-homomorphism from the cross sectional $\mathrm{C}^*$-algebra $\mathrm{C}^*(\mathcal{B})$ into $A\rtimes_\theta G$ which is the identity on each $B_{gN}$.

On the other hand, $\mathrm{C}^*(\mathcal{B})$ can also be seen as a $G$-graded $\mathrm{C}^*$-algebra, and the Fell bundle over $G$ associated to this grading is isomorphic to the Fell bundle of $A\rtimes_\theta G$. From these considerations, we conclude that $A\rtimes_\theta G$ is isomorphic to $\mathrm{C}^*(\mathcal{B})$.

Because of \cite[Proposition 4.4]{vb1998stability}, it holds that $B_eN$ is stable.

Due to \cite[Theorem 27.11]{exel}, it follows that there is a partial action $\beta$ of $\frac{G}{N}$ on  $A\rtimes_{\theta}N$ such that
$$(A\rtimes_{\theta}N)\rtimes_\beta\frac{G}{N}\cong \mathrm{C}^*(\mathcal{B})\cong A\rtimes_\theta G.$$

\end{proof}

\bibliographystyle{siam}
\bibliography{bibliografia}
\newcommand{\Addresses}{{
  \bigskip
  Department of Mathematical Sciences, University of Copenhagen,\\
Universitetsparken 5, DK-2100 Copenhagen, Denmark\par\nopagebreak
  \textit{E-mail address:} \texttt{duduscarparo@gmail.com}

}}

\Addresses

\end{document}